\documentclass{amsart}
\usepackage{amssymb}
\usepackage{amsthm}
\usepackage{amsmath}
\usepackage[mathscr]{eucal}
\usepackage{graphicx}

\DeclareMathOperator{\Age}{Age}

\newcommand{\concat}{%
  \mathord{
    \mathchoice
    {\raisebox{1ex}{\scalebox{.7}{$\frown$}}}
    {\raisebox{1ex}{\scalebox{.7}{$\frown$}}}
    {\raisebox{.7ex}{\scalebox{.5}{$\frown$}}}
    {\raisebox{.7ex}{\scalebox{.5}{$\frown$}}}
  }
}

\begin{document}

\title{Rainbow Ramsey Simple Structures}

\author[N. Dobrinen]{Natasha Dobrinen}
\thanks{The first author was supported by National Science Foundation Grant  DMS-1301665} 
\address{N. Dobrinen: Department of Mathematics\\
University of Denver \\
2280 S. Vine St.\\ Denver, CO \ 80208 U.S.A.}
\email{natasha.dobrinen@du.edu}
\urladdr{http://web.cs.du.edu/~ndobrine}

\author[C. Laflamme] {Claude Laflamme}
\thanks{The second author was supported by NSERC
of Canada Grant \# 690404}
\address{C. Laflamme: University of Calgary, Department of Mathematics and Statistics, Calgary, Alberta, Canada T2N 1N4} 
\email {laf@math.ucalgary.ca} 

\author [N. Sauer]{Norbert Sauer}
\thanks{The third author was supported by NSERC of
Canada Grant \# 691325} 
\address{N. Sauer: University of Calgary, Department of
Mathematics and Statistics, Calgary, Alberta, Canada T2N 1N4} \email{
nsauer@math.ucalgary.ca} 


\begin{abstract}

A relational structure $\mathrm{R}$ is {\em rainbow Ramsey} if for
every finite induced substructure $\mathrm{C}$ of $\mathrm{R}$ and
every colouring of the copies of $\mathrm{C}$ with countably many
colours, such that each colour is used at most $k$ times for a fixed
$k$, there exists a copy $\mathrm{R}^\ast$ of $\mathrm{R}$ so that the copies 
of $\mathrm{C}$ in $\mathrm{R^\ast}$ use each colour at most once.

We show that certain ultrahomogenous binary relational structures, for
example the Rado graph, are rainbow Ramsey. Via compactness this then
implies that for all finite graphs $\mathrm{B}$ and $\mathrm{C}$ and
$k \in \omega$, there exists a graph $\mathrm{A}$ so that for every
 colouring of the copies of $\mathrm{C}$ in $\mathrm{A}$ such that
each colour is used at most $k$ times, there exists a copy
$\mathrm{B}^\ast$ of $\mathrm{B}$ in $\mathrm{A}$ so that the copies
of $\mathrm{C}$ in $\mathrm{B^\ast}$ use each colour at most once.

\end{abstract}

\maketitle

\newcommand{\mnl} {\medskip\noindent}
\newcommand{\bqed} {\ \ \ \vrule height5pt width4pt depth3pt}
\newcommand{\sqed} {\ \ \ \vrule height3pt width2pt depth1pt}
\newcommand{\snl}{\smallskip\noindent}
\newcommand{\bnl}{\bigskip\noindent}
\newcommand{\nl}{\hfil\break}
\newcommand{\Forb}{\mathrm{Forb}}
\newcommand{\sforb}{\mathrm{SForb}}
\newcommand{\Up}{\mathrm{Up}}
\newcommand{\age}{{\rm{age}}}
\newcommand{\rem}{\rm{rm}}
\newcommand{\univ}{\rm{univ}}
\newcommand{\skel}{\rm{skel}}
\newcommand{\tskel}{\rm{tskel}}
\newcommand{\rskel}{\mathrm{rskel}}
\newcommand{\struct}{\mathrm{struct}}
\newcommand{\orb}{\mathbf{orb}}
\newcommand{\agree}{\mathrm{agree}}
\newcommand{\domain}{\mathrm{domain}}
\newcommand{\bound}{{\mathrm{bound}}}
\newcommand{\replace}{\mathrm{replace}}
\newcommand{\cut}{\mathrm{cut}}
\newcommand{\rmpairs}{\mathrm{rmpairs}}
\newcommand{\rmsingles}{\mathrm{rmsingles}}
\newcommand{\Bound}{\mathrm{Bound}}
\newcommand{\cont}{{\mathrm{cont}}}
\newcommand{\tourn}{\mathrm{tourn}}
\newcommand{\ord}{\mathrm{ord}}
\newcommand{\broken}{\mathrm{broken}}
\newcommand{\join}{\mathrm{join}}
\newcommand{\D}{{\mathrm D}}
\newcommand{\pprime}{{\prime\prime}}
\newcommand{\ppprime}{{\prime\prime\prime}}
\newcommand{\closure}{\mathrm{closure}}
\newcommand{\G}{{\mathrm G}}
\newcommand{\Ha}{{\mathrm H}}
\newcommand{\Ht}{\Ha_{\mathcal T}}
\newcommand{\A}{{\mathrm A}}
\newcommand{\F}{{\mathrm F}}
\newcommand{\Bg}{{\mathrm B}}
\newcommand{\K}{{\mathrm K}}
\newcommand{\T}{{\mathrm T}}
\newcommand{\Sa}{{\mathrm S}}
\newcommand{\aut}{\mathrm{aut}}
\newcommand{\J}{{\mathrm J}} 
\newcommand{\Q}{{\mathrm Q}}
\newcommand{\Pa}{{\mathrm P}}
\newcommand{\X}{{\mathrm X}}
\newcommand{\C}{{\mathrm C}}
\newcommand{\Y}{{\mathrm Y}}
\newcommand{\U}{{\mathrm U}}
\newcommand{\V}{{\mathrm V}}
\newcommand{\W}{{\mathrm W}}
\newcommand{\Z}{{\mathrm Z}}
\newcommand{\La}{{\mathrm L}}
\newcommand{\cp}{\mathrm{CP}}
\newcommand{\tto}{\twoheadrightarrow}
\newcommand{\rel}{\mathrm{rel}}
\newcommand{\graph}{\mathrm{graph}}
\newcommand{\norm}{\mathrm{norm}}
\newcommand{\reduce}{{\mathrm{reduce}}}
\newcommand{\Is}{{\mathrm{Is}}}
\newcommand{\R}{\mathrm{R}}
\newcommand{\B}{\mathrm{B}}
\newcommand{\Sims} {\mathrm{Sims}}
\newcommand{\Sim} {\mathrm{Sim}}
\newcommand{\Can} {\mathrm{Can}}
\newcommand{\Cans} {\mathrm{Cans}}
\newcommand{\Str} {\mathrm{Str}}
\newcommand{\structure}{\mathrm{struct}}

\newcommand{\ra}{\rightarrow}
\newcommand{\Lla}{\Longleftarrow}
\newcommand{\Lra}{\Longrightarrow}
\newcommand{\Llra}{\Longleftrightarrow}
\newcommand{\lgl}{\langle}
\newcommand{\rgl}{\rangle}

\newcommand{\al}{\alpha}
\newcommand{\om}{\omega}
\newcommand{\vp}{\varphi}
\newcommand{\sse}{\subseteq}
\newcommand{\contains}{\supseteq}
\newcommand{\forces}{\Vdash}
\newcommand{\levels}{\mathrm{levels}}

\newcommand{\Fra}{Fra\"{\i}ss\'e\  }

\newcommand{\down}{~{\hskip-3pt\downarrow\hskip-3pt}~}

\newcommand{\core}{\hbox{core\ }}
\newcommand{\one}{\mathop {1}\limits^{\circ}}
\newcommand{\edge}{\hbox{$\circ\mkern-8mu -\mkern-8mu\circ\mkern2mu$}}
\newcommand{\arrow}{\twoheadrightarrow}
\newcommand{\id}{\hbox{id}}

\newtheorem{thm}{Theorem}[section]
\newtheorem{lem}{Lemma}[section]
\newtheorem{cor}{Corollary}[section]
\newtheorem{ass}{Assumption}[section]
\theoremstyle{definition}
\newtheorem{defin}{Definition}[section]
\newtheorem{example}{Example}[section]
\newtheorem{fact}{Fact}[section]
\newtheorem{prop}[thm]{Proposition} 
\newtheorem{question}[thm]{Question}
\newtheorem{questions}[thm]{Questions}

\section{Introduction}\label{sec.intro}

There is an extensive literature concerning colouring problems of the
following type: Given conditions on the colouring function conclude
that the restriction of the colouring function to a particular subset
of its range is injective. In particular, this has been been widely
studied for finite graphs, and then mostly complete graphs, and edge
colourings. See \cite{FMO} and \cite{KL} for two survey papers.
Many of those problems and results are in analogy to standard Ramsey
type problems and results, for which then the restriction of the
colouring function to a particular subset of its range is asked to be
constant. Those investigations have given rise to notions like
anti-Ramsey numbers and restricted Ramsey numbers etc; and then
results finding anti-Ramsey numbers for certain pairs of graphs,
finding upper bounds, complexity, asymptotic behavior of the
anti-Ramsey numbers.  As is outlined in the two surveys cited above,
problems in the context of rainbow Ramsey have been studied
extensively for $k$-bounded colourings of pairs of natural numbers,
from finding exact numbers, to finding growth rates, to investigations
of the relative strength of the statement to other Ramsey properties;
see \cite{CM} for relative strength investigation.  There are a few
articles extending the work to hypergraphs and a few to infinite
graphs. See \cite{LVW} and \cite{HT}.

However, the literature has completely missed, for colouring of
arbitrary finite substructures, finding rainbow copies of the Rado
graph and other ultrahomogeneous structures, as well as investigating
$k$-bounded colorings of copies of a fixed finite graph, rather than
simply edge colorings. In this work, we address this void.

For $\mathrm{B}$ a relational structure we often denote by $B$ the set
of its elements and sometimes use $\mathrm{B}$ to denote the structure
as well as the set of its elements.  For $S \subseteq B$ let
$\mathrm{B}\down S$ be the substructure of $\mathrm{B}$ induced by
$S$. For $\mathrm{B}$ and $\mathrm{C}$ relational structures denote by
\[
\binom{\mathrm{B}}{\mathrm{C}}
\]
the set of induced substructures of $\mathrm{B}$ isomorphic to
$\mathrm{C}$. The elements of $\binom{\mathrm{B}}{\mathrm{B}}$ are the
{\em copies of $B$ in $B$}. An {\em embedding} of $\mathrm{B}$ is an
isomorphism of $\mathrm{B}$ to a copy of $\mathrm{B}$ in
$\mathrm{B}$.

Let $\mathrm{A}$ and  $\mathrm{C}$ be relational structures and
$k$ a natural number.  A function $\gamma:
\binom{\mathrm{A}}{\mathrm{C}}\to \omega$ is called {\em $k$-bounded} if
$|\gamma^{-1}(n)|\leq k$ for every $n\in \omega$. A weaker notion and
more appropriate for the arguments in this paper is that of of a
$k$-delta function. 
By a {\em $\delta$-system of copies} of
$\mathrm{C}$ in $\mathrm{A}$ we mean a subset $\mathcal{S}$ of
$\binom{\mathrm{A}}{\mathrm{C}}$ for which
$X\setminus\bigcap_{\mathrm{Y}\in \mathcal{S}}Y$ is a singleton set
for all $\mathrm{X}\in \mathcal{S}$. 
(This is a special case of the more general set-theoretic notion of $\delta$-system.)
Then we call a function $\gamma:
\binom{\mathrm{A}}{\mathrm{C}}\to \omega$ is {\em $k$-delta} if there
is no $\delta$-system $\mathcal{S}\subseteq
\binom{\mathrm{A}}{\mathrm{C}}$ having $k+1$ elements so that $\gamma$
is constant on $\mathcal{S}$. Clearly every $k$-bounded function is
$k$-delta.

The rainbow Ramsey properties we are interested in are defined as follows.

\begin{defin}\label{defin:rainbowRa}   Let $\mathrm{A}$, $\mathrm{B}$ and $\mathrm{C}$ be relational
structures and $k$ a natural number.

\noindent The arrow
\begin{align}
\mathrm{A}\underset{k\text{-delta}}{\xrightarrow{\mathrm{rainbow}}}\big(\mathrm{B}\big)^\mathrm{C} 
\end{align}
means that for every $k$-delta colouring $\gamma:
\binom{\mathrm{A}}{\mathrm{C}}\to \omega$, there exists a
$\mathrm{B}^\ast\in \binom{\mathrm{A}}{\mathrm{B}}$ so that $\gamma$
is one-to-one on $\binom{\mathrm{B^\ast}}{\mathrm{C}}$. \\
The arrow
\begin{align}
\mathrm{A}\underset{k\text{-bdd}}{\xrightarrow{\mathrm{rainbow}}}\big(\mathrm{B}\big)^\mathrm{C} 
\end{align}
means that for every $k$-bounded colouring $\gamma:
\binom{\mathrm{A}}{\mathrm{C}}\to \omega$, there exists a
$\mathrm{B}^\ast\in \binom{\mathrm{A}}{\mathrm{B}}$ so that $\gamma$
is one-to-one on $\binom{\mathrm{B^\ast}}{\mathrm{C}}$.
\end{defin}

Thus 
$\mathrm{A}\underset{k\text{-delta}}{\xrightarrow{\mathrm{rainbow}}}\big(\mathrm{B}\big)^\mathrm{C}
$ implies
$\mathrm{A}\underset{k\text{-bdd}}{\xrightarrow{\mathrm{rainbow}}}\big(\mathrm{B}\big)^\mathrm{C}$.
\vskip 6pt

A countably infinite relational structure $\mathrm{R}$ is {\em rainbow
Ramsey} if for every finite substructure $\mathrm{C}$ and every $k\in
\omega$ the relation
$\mathrm{R}\underset{k\text{-bdd}}{\xrightarrow{\mathrm{rainbow}}}\big(\mathrm{R}\big)^\mathrm{C}$
holds; similarly  $\mathrm{R}$ is {\em delta rainbow
Ramsey} if for every finite substructure $\mathrm{C}$ and every $k\in
\omega$ the relation
$\mathrm{R}\underset{k\text{-delta}}{\xrightarrow{\mathrm{rainbow}}}\big(\mathrm{R}\big)^\mathrm{C}$
holds. A class $\mathscr{F}$ of finite relational structures is {\em
rainbow Ramsey} if for all structures $\mathrm{B}$ and $\mathrm{C}$ in
$\mathscr{F}$ and every $k\in \omega$, there exists a structure
$\mathrm{A}\in \mathscr{F}$ for which relation (2) holds; similarly
for {\em delta rainbow Ramsey}.  If $\mathrm{R}$ is a countably
infinite (delta) rainbow Ramsey relational structure, then the
class of finite structures isomorphic to an induced substructure of
$\mathrm{R}$, the age of $\mathrm{R}$, $\Age(\mathrm{R})$, is also
(delta) rainbow Ramsey; see Corollary  \ref{cor:infinite->finite}.

Relational structures $\mathrm{A}$ and $\mathrm{B}$ are called {\em
equimorphic}, or said to be {\em siblings}, if there exists an
isomorphic injection of $\mathrm{A}$ into $\mathrm{B}$ and an
isomorphic injection of $\mathrm{B}$ into $\mathrm{A}$. 
Note that if
$\mathrm{A}$ and $\mathrm{B}$ are siblings then for every
$\mathrm{A}^\ast \in
\binom{\mathrm{A}}{\mathrm{A}}$ the set
$\binom{\mathrm{A}^\ast}{\mathrm{B}}\not=\emptyset$ and for every
$\mathrm{B}^\ast\in \binom{\mathrm{A}}{\mathrm{B}}$ the set
$\binom{\mathrm{B}^\ast}{\mathrm{A}}\not=\emptyset$. 
If $\mathrm{A}$
and $\mathrm{B}$ are equimorphic relational structures, then
$\mathrm{A}$ is (delta) rainbow Ramsey if and only if $\mathrm{B}$ is
(delta) rainbow Ramsey; see Lemma \ref{lem:siblingsrain}.

A binary relational structure is {\em simple} if for each of its
symmetric relations the structure is a simple graph and for each of
its directed relations the structure is the orientation of a simple
graph; see Definition \ref{defin:simplestr} for a more precise
definition.  The Rado graph is an example of a simple binary structure
as is every simple graph.  A structure is {\em ultrahomogeneous}, also
often called simply {\em homogeneous}, if every isomorphism between
finite substructures extends to an automorphism of the entire
structure. See \cite{F} for an introduction to homogeneous structures.
We prove that simple ultrahomogeneous countable binary relational
structures are (delta) rainbow Ramsey, as are other simple binary
relational structures; see Theorem
\ref{thm:final} for a more precise statement. It follows that the age
of any (delta) rainbow Ramsey structure is also (delta)
rainbow Ramsey, see Corollary \ref{cor:infinite->finite}. 
That such classes of finite simple binary relational
structures are rainbow Ramsey also follows via a
similar construction as given here from the Ne\v set\v ril-R\"odl
partition theorem; see \cite{NR}.
In fact,  any Fra\"{i}ss\'{e} class of finite relational structures with free amalgamation and  the Ramsey property is rainbow Ramsey, by a similar argument.

We conclude the introduction by pointing out that,
while rainbow Ramsey holds for $k$-bounded colorings of the natural numbers as an easy consequence of Ramsey's Theorem, and while
for uncountable sets rainbow Ramsey results are  independent of ZFC (see \cite{Todorcevic83} and \cite{Abraham/Cummings/Smyth07}),  
ours is the first work showing that  homogeneous relational structures without the Ramsey property may still be rainbow Ramsey.

The authors wish to thank Jaukub Jasi\'{n}ski for asking the question of whether rainbow Ramsey might hold for  the Rado graph, which led to this work.

\section{Trees of sequences}

Let $\mathfrak{T}_\omega$ be the tree consisting of all finite
sequences with entries in $\omega$.  The order of the tree
$\mathfrak{T}$, denoted by $\subseteq$, is given by sequence
extension.  That is $x=\langle x(0),x(1),\dots, x(n-1)\rangle\subseteq
y=\langle y(0),y(1),\dots, y(m-1)\rangle$ if $n\leq m$ and $x(i)=y(i$)
for all $i\in n$. Let $S\subseteq \mathfrak{T}_\omega$.  The sequence
$y\in S$ is an {\em immediate $S$-successor} of the sequence $x\in S$
if $x\subseteq y$ and and there is no sequence $z\in S$ with
$x\subsetneq z\subsetneq y$. The $S$-degree of $x\in S$ is the number
of immediate $S$-successors of $x$.  The length $|x|$ of the sequence
$x=\langle x(0),x(1),\dots, x(n-1)\rangle$ is $n$.  The {\em
intersection} $x\wedge y$ of two sequences is the longest sequence $z$
with $z\subseteq x$ and $z\subseteq y$. For $R\subseteq
\mathfrak{T}_\omega$ let $\levels(R)=\{|x| : x\in R\}$. For $|x|<|y|$
the number $y(|x|)$ is the {\em passing number} of $y$ at $x$.


We write $x\prec y$ if $x$ and $y$ are incomparable under $\subseteq$
and if
\[
x(|x\wedge y|)<y(|x\wedge y|).  
\]
That is, for two incomparable sequences $x$ and $y$, if the passing
number of $x$ is smaller than the passing number of $y$ at their
intersection then $x\prec y$. For $\subseteq$ incomparable sequences
the relation $\prec$ agrees with the lexicographic order on the tree
$\mathfrak{T_d}$.

As in the third paragraph of the preliminaries section of \cite{S} and
Definition~4.1 of \cite{S}, we define:
\begin{defin}\label{defin:omegatr}
A subset $T$ of $\mathfrak{T}_\omega$ is an $\omega$-tree if it is not
empty, closed under initial segments, has no endpoints and the
$T$-degree of every sequence $x\in T$ is finite.
\end{defin} 

For $2\leq \mathfrak{d}\in \omega$ let $\mathfrak{T_d}\subseteq
\mathfrak{T}_\omega$ be the tree consisting of all finite sequences
with entries in $\mathfrak{d}$.  Then $\mathfrak{T_d}$ is an
$\omega$-tree and actually a wide $\omega$-tree according to
Definition~4.1 of \cite{S}.  The root of $\mathfrak{T_d}$ is the empty
sequence $\emptyset$. 

Let $S\subseteq \mathfrak{T_d}$ be a set of sequences.  Following
Definitions 2.2 and 3.2 of \cite{S} we define: The set $S$ is {\em
transversal} if no two different elements of $S$ have the same length
and it is an {\em antichain} if no two different elements of $S$ are
ordered under the $\subseteq$ relation. The closure of $S$ is the set
of intersections of elements of $S$. A subset $V\subseteq
\mathfrak{T_d}$ is cofinal if for every $x\in \mathfrak{T_d}$ there is
a $v\in V$ with $x\subseteq v$.

The set $S$ is {\em diagonal} if it is an antichain and the closure of
$S$ is transversal and the degree of every element of closure $S$
within the closure of $S$ is at most 2.  The set $S$ is {\em strongly
diagonal} if it is diagonal and if for all $x,y,z\in S$ with
$x\not=y$:
\begin{enumerate} 
\item If $|x\wedge y|<|z|$ and $x\wedge y\not\subseteq z$ then $z(|x\wedge y|)=0$. 
\item $x(|x\wedge y|)\in \{0,1\}$. 
\end{enumerate}

\begin{defin}\label{defin:sim}
A function $f$ mapping a set $R\subseteq \mathfrak{T_d}$ of sequences
into $\mathfrak{T_d}$ has the properties of Order, Level, Level-imp, Pnp (passing
number preservation), Pnp-strong or Lexico respectively, if for all
$x,y,z,u\in R$:
\begin{description}
\item[Order:] \hskip 40pt $x\wedge y\subseteq z\wedge u$ if and only if $f(x)\wedge f(y)\subseteq f(z)\wedge f(u)$. 
\item[Level:] \hskip 42pt$|x\wedge y|<|z\wedge u|$ if and only if $|f(x)\wedge f(y)|<|f(z)\wedge f(u)|$.
\item[Level-imp:] \hskip 19pt$|x\wedge y|<|z\wedge u|$ implies $|f(x)\wedge f(y)|<|f(z)\wedge f(u)|$.
\item[Pnp:] \hskip 47pt If  $|z|>|x|$ then $z(|x|)=f(z)(|f(x)|)$.
\item[Pnp-strong:]\hskip 10pt  If  $|z|>|x\wedge y|$ then $z(|x\wedge y)|=f(z)(|f(x)\wedge f(y)|)$.
\item[Lexico:] \hskip 34 pt If $x\prec y$ then $f(x)\prec f(y)$. 
\end{description}   
\end{defin}

According to Definition 3.1 of \cite{S}, we define:
\begin{defin}\label{defin: strongsim}
Let $R$ and $S$ be two subsets of $\mathfrak{T}_\omega$. A bijection
$f$ of $R$ to $S$ is a {\em strong similarity} of $R$ to $S$ if it
satisfies the properties of Order and Level and Pnp-strong; (and also
Lexico which is implied by Pnp-strong).

The sets $R$ and $S$ are strongly similar, $R\stackrel{s}{\sim} S$, if
there is a strong similarity, (it will be unique), of $R$ to $S$.

Then for $T$ a set of sequences, the set $\Sims_T(R)$ is the
equivalence class of all subsets $S$ of $T$ with $R\stackrel{s}{\sim}
S$.
\end{defin}

Observe that a strong similarity maps sequences of the same length to
sequences of the same length. It can be viewed as ``stretching" or in
the inverse then ``compressing" length of sequences, but preserving
all shapes and passing numbers.

\section{Binary structures encoded by  $\mathfrak{T_d}$}

 \begin{defin}\label{defin:simplestr} A binary relational structure
 $\mathrm{R}$ on a set $R$ with binary relations $E_0, E_1, \dots,
 E_{n-1}$ and $n\geq 2$ is {\em simple} if: 

\begin{enumerate} 

\item Each of the binary relations $E_i$ is irreflexive.  
\item There is a number $m\leq n$, the {\em symmetry number} of $\mathrm{R}$ so that:
 \begin{enumerate} 
	\item $E_i(x,y)$ implies $E_i(y,x)$ for all $x,y\in  R$ and all $i\in m$.  
	\item $E_i(x,y)$ implies $\neg E_i(y,x)$ for  all $x,y\in R$ and all $m\leq i\in n$.  
  \end{enumerate} 
\item For all $i,j\in n$ with $i\not=j$: $E_i(x,y)$ implies $\neg E_j(x,y)$ and  $\neg E_j(y,x)$.  
\item For all $(x,y)$, there is an $i\in n$ with $E_i(x,y)$.
\end{enumerate} 
\end{defin} 

That is, the set $R$ together with only one of the relations $E_i$ is
a simple graph if $i\in m$, and is an orientation of a simple graph if
$m\leq i\in n$.  Any two different such graphs do not have overlapping
edges. Note that every binary relational structure having only
irreflexive relations can be encoded as a simple binary relational
structure.
 
For $n$ and $m$ in $\omega$ and $m\leq n$ let $\mathscr{S}(n,m)$
denote the class of all simple binary relational structures with set
of binary relations $\{E_0, E_1, \dots, E_n\}$ and symmetry number
$m$. Let $\mathscr{A}(n,m)\subseteq \mathscr{S}(n,m)$ be the subclass
of the finite structures in $\mathscr{S}(n,m)$.  The class
$\mathscr{A}(n,m)$ is an age with amalgamation whose \Fra limit is
denoted by $\mathbb{U}(n,m)$. The ultrahomogeneous structures of the
form $\mathbb{U}(n,m)$ are the {\em random simple binary
ultrahomogeneous structures}.  Note that then $\mathbb{U}(2,2)$ is the
random graph. (It has adjacent and non-adjacent two element subsets as
sets of relations.)  Also, $\mathbb{U}(1,0)$ is the random Tournament,
$\mathbb{U}(2,1)$ is the random oriented graph, $\mathbb{U}(1,1)$ is
the complete graph $\mathrm{K}_\omega$, $\mathbb{U}(0,0)$ is the
relational structure having no relations, $\mathbb{U}(3,3)$ may be
seen as a graph having heavy and light edges and $\mathbb{U}(3,2)$ as
a graph having symmetric and oriented edges.  The cases $(n,m)=(0,0)$
and $(n,m)=(1,1)$ do not quite fit into the general notational
framework developed subsequently and will be dealt with later on. Note
that for $m\leq n$ the inequality $(0,0)\not=(n,m)\not=(1,1)$ holds if
and only if $m+2(n-m)\geq 2$.

Let $m\leq n\in \omega$ and $\mathfrak{d}=m+2(n-m)\geq 2$.  The tree
$\mathfrak{T_d}$ encodes a simple binary relational structure
$\mathbb{T}(n,m)\in \mathscr{S}(n,m)$ on the set of sequences in
$\mathfrak{T_d}$, in which for all $x,y\in \mathfrak{T_d}$ with
$x\not=y$: \begin{enumerate} \item If $|x|<|y|$ then let for the
passing number $p=y(|x|)$: \begin{enumerate} \item If $p\in m$ then
$E_p(x,y)$ and $E_p(y,x)$.  \item If $p\geq m$ and $p-m$ even then
$E_{m+\frac{p}{2}}(x,y)$.  \item If $p\geq m$ and $p-m$ odd then
$E_{m+\frac{p-1}{2}}(y,x)$.  \end{enumerate} \item If $|x|=|y|$ and
$x\prec y$ then if $m=0$: $E_0(x,y)$.  If $m>0$: $E_0(x,y)$ and
$E_0(y,x)$.  \end{enumerate} For $\mathrm{R}$ a relational structure
and $F$ a subset of $R$, the set of elements of $\mathrm{R}$, let
$\mathrm{R}\down F$ be the substructure of $\mathrm{R}$ induced by
$F$.  Note that if $f$ is a similarity of a subset $F$ of
$\mathfrak{T_d}$ to a subset $G$ of $\mathfrak{T}_d$ then $f$ is an
isomorphism of $\mathbb{T}(n,m)\down F$ to $\mathbb{T}(n,m)\down G$.
 
In accordance with Definition 3.3 of \cite{S}: 
\begin{defin}\label{defin:strdiagonalization}
An injection $f: \mathfrak{T_d}\to \mathfrak{T_d}$ is a {\em strong
diagonalization} of $\mathfrak{T_d}$ if the image of $f$ is a strongly
diagonal subset of $\mathfrak{T_d}$ and if $f$ satisfies the
properties of Level-imp and Pnp and Lexico.
\end{defin}

According to Theorem 4.1 of \cite{S} there exists a strong
diagonalization of $\mathfrak{T_d}$ into $\mathfrak{T_d}$. We fix for
every $2\leq \mathfrak{d}\in \omega$ such a strong diagonalization
$\Delta_{\mathfrak{d}}$ and denote by $\mathfrak{D_\mathfrak{d}}$
the image of the strong diagonalization $\Delta_{\mathfrak{d}}$. For
$\mathfrak{d}=m+2(n-m)\geq 2$ let
$\mathbb{D}(n,m)=\mathbb{T}(n,m)\down\mathfrak{D_\mathfrak{d}}$. Note
that $\Delta_{\mathfrak{d}}$ is an isomorphism of $\mathbb{T}(n,m)$ to
$\mathbb{D}(n,m)$.  Hence $\mathbb{D}(n,m)$ is equimorphic to
$\mathbb{T}(n,m)$; but is not a homogeneous structure.

\begin{thm}\label{thm:sibUT}
Let $m\leq n\in \omega$ and $m+2(n-m)\geq 2$. The simple binary
relational structures $\mathbb{U}(n,m)$ and $\mathbb{T}(n,m)$ and
$\mathbb{D}(n,m)$ are equimorphic.  \end{thm} \begin{proof} Because
$\mathbb{T}(n,m)$ is equimorphic to $\mathbb{D}(n,m)$ it suffices to
show that $\mathbb{T}(n,m)$ and $\mathbb{U}(n,m)$ are equimorphic.

Every finite induced substructure of $\mathbb{T}(n,m)$ is an element
of $\mathscr{A}(n,m)$. Hence and because the homogeneous structure
$\mathbb{U}(n,m)$ is universal, there exists an isomorphism of
$\mathbb{T}(n,m)$ into $\mathbb{U}(n,m)$.

Let $\mathfrak{d}=m+2(n-m)$ and $V$ be a transversal and cofinal
subset of $\mathfrak{T_d}$. Let $F$ be a finite subset of $V$. Then
for every function $f: F\to \mathfrak{d}$ there exists a sequence
$y\in V$ so that the passing number $y(|x|)=f(x)$ for every $x\in
F$. This then translates to the fact that the substructure of
$\mathbb{T}(n,m)$ induced by $V$ has the mapping extension property
with respect to the age $\mathscr{A}(n,m)$. Hence the structure
induced by $V$ is homogeneous with age $\mathscr{A}(n,m)$ and hence
isomorphic to $\mathbb{U}(n,m)$.  \end{proof}

\section{Milliken}

\begin{defin}\label{defin:strong}
The set $T\subseteq \mathfrak{T_d}$ is {\em closed by levels} if $s\in
T$ whenever there exist $t\in T$ with $s\subseteq t$ and $y\in T$ with
$|s|=|y|$.  Let $T$ be a meet closed set of sequences which is also
closed by levels and let $n\in \omega+1$. The set $S\subseteq T$ is an
element of $\Str^n(T)$ if:
\begin{enumerate}
  \item $|\mathrm{levels}(S)|=n$.  
  \item $S$ is meet closed and closed by levels. 
  \item For all $s\in S$,   the degree of $s$  in $S$ is 0 or  equal to the degree of $s$ in $T$.
  \end{enumerate}
\end{defin}   
Note the following:

The tree $\mathfrak{T_d}$ is meet closed and closed by levels.  Let
$T\sse \mathfrak{T_d}$ be meet closed, closed by levels, non-empty and
having no endpoints.  Then for $S\in \Str^\omega(T)$ removing, from
every sequence $s\in S$, the entries whose indices are not in
$\levels(S)$ results in the tree $\mathfrak{T_d}$.  For every $S\in \Str^\omega(T)$
there exists a unique strong similarity $f: T\to S$.  This
similarity $f$ maps every level of $T$ into a level of $S$, preserving
the order, via length, of the levels.  Conversely, for every strong
similarity $f$ of $T$ into $T$, the set $f[T]\in \Str^\omega(T)$.  If $R\in
\Str^\omega(S)$ and $S\in \Str^\omega(T)$ then $R\in \Str^\omega(T)$.
For $T\in \Str^\omega(\mathfrak{T_d})$ let $T^\ast:= \{x\in
\mathfrak{T_d}\mid \exists y\in T\, (x\subseteq y)\}$. Then $T^\ast$
is an $\omega$-tree and if $T^\ast=S^\ast$ then $T=S$.

 Hence the following Lemma \ref{lem:Milik} is a direct consequence of Theorem 5.2 of~\cite{S}. 
\begin{lem} \label{lem:Milik}
Let $T=f[\mathfrak{T_d}]$ for $f$ a strong similarity of\/ $\mathfrak{T_d}$ to $T\subseteq \mathfrak{T_d}$.  Let  $F$ be a finite meet closed subset of $T$ and  $m\in \omega$. Then for any colouring $\gamma: \Sims_T(F)\to m\in \omega$ there exists a strong similarity function $f$ of\/ $T$ to a tree  $S\subseteq T$  so that $\gamma$ is constant on $\Sims_S(F)$.   
\end{lem}

The following thus follows from the fact that if $F$ and $G$ are two
finite and strongly diagonal sets with $F\stackrel{s}{\sim} G$ and
with $\closure(F)=\closure(G)$ then $F=G$.

\begin{cor} \label{cor:Milik2}
Let $T=f[\mathfrak{T_d}]$ for $f$ a strong similarity of\/ $\mathfrak{T_d}$ to $T\subseteq \mathfrak{T_d}$.  Let  $F$ be a finite and strongly diagonal   subset of $T$ and  $m\in \omega$. Then for any colouring $\gamma: \Sims_T(F)\to m\in \omega$ there exists a strong similarity function $f$ of\/ $T$ to a tree  $S\subseteq T$  so that $\gamma$ is constant on $\Sims_S(F)$.   
\end{cor}

By repeated application of the above Lemma \ref{lem:Milik} and Corollary \ref{cor:Milik2} we obtain:

\begin{thm}\label{thm:fundamental}
Let $\{F_i\mid i\in p\in \omega\}$ be a finite set of finite meet
closed subsets of\/ $\mathfrak{T_d}$ so that $\neg (F_i
\stackrel{s}{\sim} F_j)$ for $i\not=j$.  Let $m_i\in \omega$ for all
$i\in p$. Then for any set $\gamma_i: \Sims_{\mathfrak{T_d}}(F_i)\to
m_i$ of colouring functions there exists a strong similarity function
$f$ of\/ $\mathfrak{T_d}$ to a tree $T\subseteq \mathfrak{T_d}$ so
that each one of the colouring functions $\gamma_i$ is constant on
$\Sims_T(F_i)$.
\end{thm}

\begin{cor}\label{cor:fundamental}
Let $\{F_i\mid i\in p\in \omega\}$ be a finite set of strongly
diagonal subsets of\/ $\mathfrak{T_d}$ so that $\neg (F_i
\stackrel{s}{\sim} F_j)$ for $i\not=j$.  Let $m_i\in \omega$ for all
$i\in p$. Then for any set $\gamma_i: \Sims_{\mathfrak{T_d}}(F_i)\to
m_i$ of colouring functions there exists a strong similarity function
$f$ of\/ $\mathfrak{T_d}$ to a tree $T\subseteq \mathfrak{T_d}$ so
that each one of the colouring functions $\gamma_i$ is constant on
$\Sims_T(F_i)$.
\end{cor}

Hence we can apply these results to  $\mathfrak{D_d}$.

\begin{thm}\label{lem:beginbas}
Let $\{F_i\mid i\in p\in \omega\}$ be a finite set of finite subsets
of\/ $\mathfrak{D_d}$ so that $\neg (F_i \stackrel{s}{\sim} F_j)$ for
$i\not=j$.  Let $m_i\in \omega$ for all $i\in p$. Then for any set
$\gamma_i: \Sims_{\mathfrak{D_d}}(F_i)\to m_i$ of colouring functions
there exists a strong similarity function $g$ of\/ $\mathfrak{D_d}$ to
a tree $D\subseteq \mathfrak{D_d}$ so that each one of the colouring
functions $\gamma_i$ is constant on $\Sims_D(F_i)$.
\end{thm}

\begin{proof}
For every $i\in p$ let $\delta_i: \Sims_{\mathfrak{T_d}}(F_i)\to m_i$ be given by $\delta_i(F_i')=\gamma_i(\Delta_{\mathfrak{d}}(F_i'))$ for $F_i'\in \Sims_{\mathfrak{T_d}}(F_i)$.  Observing that $F_i\stackrel{s}{\sim}F_i'$ if and only if $F_i\stackrel{s}{\sim}\Delta_{\mathfrak{d}}(F_i')$ because $F_i$ and hence $F_i'$ are strongly diagonal. 

According to Corollary \ref{cor:fundamental} there exists a strong
similarity function $f$ of\/ $\mathfrak{T_d}$ to a tree $T\subseteq
\mathfrak{T_d}$ so that each one of the colouring functions $\delta_i$
is constant on $\Sims_T(F_i)$. The restriction of $\Delta_k$ to a
strongly diagonal subset $S$ of $\mathfrak{T_d}$ is a strong
similarity of $ S$. Hence the function $g$, which is the restriction
of $\Delta_{\mathfrak{d}}\circ f$ to $\mathfrak{D_d}$ is a strong
similarity of $\mathfrak{D_d}$ to a subset $D$ of $\mathfrak{D_d}$.
\end{proof}

Let $n\geq 2$ and $\mathfrak{d}=m+2(n-m)$. Let $\mathrm{R}$ and
$\mathrm{S}$ be two finite induced substructures of $\mathbb{D}(n,m)$
with domains $R, S \subseteq \mathcal{D}_{\mathfrak{d}}$, the set of
elements of $\mathrm{R}$ and $\mathrm{S}$ respectively. That is
$\mathrm{R}=\mathbb{D}(n,m)\down R$ and
$\mathrm{S}=\mathbb{D}(n,m)\down S$. Then $\mathrm{R}$ and
$\mathrm{S}$ are {\em strongly similar} (as substructures), written
again $\mathrm{R}\stackrel{s}{\sim}\mathrm{S}$, if their domains are
{\em strongly similar}, that is $R\stackrel{s}{\sim} S$.  Note that if
$\mathrm{R}\stackrel{s}{\sim}\mathrm{S}$ then the strong similarity of
$\mathrm{R}$ to $\mathrm{S}$ is an isomorphism of the binary structure
$\mathrm{R}$ to the binary structure $\mathrm{S}$.  For $\mathrm{S}$
an induced substructure of $\mathbb{D}(n,m)$ and $\mathrm{F}$ a finite
induced substructure of $\mathrm{S}$ let
$\Sims_{\mathrm{S}}(\mathrm{F})$ be the set of all induced
substructures $\mathrm{F}'$ of $\mathrm{S}$ with
$\mathrm{F}\stackrel{s}{\sim} \mathrm{F}'$.  Then
$\Sims_\mathrm{S}(\mathrm{F})$ is the set of all $\mathrm{F}'\in
\binom{\mathrm{S}}{\mathrm{F}}$ with
$\mathrm{F}\stackrel{s}{\sim}\mathrm{F}'$. Hence:

\begin{cor}\label{cor:beginbas}
Let $n\geq 2$ and $\mathfrak{d}=m+2(n-m)$. Let $\{\mathrm{F}_i\mid
i\in p\in \omega\}$ be a finite set of finite induced substructures
of\/ $\mathbb{D}(n,m)$ so that $\neg (\mathrm{F}_i \stackrel{s}{\sim}
\mathrm{F}_j)$ for $i\not=j$.  Let $m_i\in \omega$ for all $i\in
p$. \\
Then for any set $\gamma_i: \Sims_{\mathbb{D}(n,m)}(\mathrm{F}_i)\to
m_i$ of colouring functions there exists a strong similarity function
$g$ of\/ $\mathbb{D}(n,m)$ to a copy $\mathrm{D}$ of $\mathbb{D}(n,m)$
so that each one of the colouring functions $\gamma_i$ is constant on
$\Sims_{\mathrm{D}(\mathrm{F}_i})$.
\end{cor}

\section{ $\mathbb{T}(n,m)$ and $\mathbb{U}(n,m)$ and $\mathbb{D}(n,m)$ are rainbow Ramsey.}
 
In this section we show that simple relational structures are
$k$-delta rainbow Ramsey for any $k$, and thus rainbow Ramsey. 

First we show that it is sufficient to show the result for $k=2$.

\begin{lem}\label{lem:2bdd_enough}
Let $\mathrm{R}$ be a relational structure and let
$\C\in\Age(\mathrm{R})$.\\ If \/
$\mathrm{R}\underset{2\text{-delta}}{\xrightarrow{\mathrm{rainbow}}}\big(\mathrm{R}\big)^\mathrm{C}$,
then
$\mathrm{R}\underset{k\text{-delta}}{\xrightarrow{\mathrm{rainbow}}}\big(\mathrm{R}\big)^\mathrm{C}$
for all $k\ge 2$.
\end{lem}

\begin{proof}
Let  $\C\in\Age(\mathrm{R})$.
The proof is by induction on $k\ge 2$, where the base case is given by the hypothesis.

Thus assume that  $\mathrm{R}\underset{k\text{-delta}}{\xrightarrow{\mathrm{rainbow}}}
\big(\mathrm{R}\big)^\mathrm{C}$. 
Let $f:{\mathrm{R}\choose \C}\ra \om$ be a $k+1$-delta colouring.
Note that if $\mathcal{S}$ is a $\delta$-system of copies of $C$, then
$|\bigcap_{Y\in\mathcal{S}}Y|=|C|-1$.  Thus, each maximal
$\delta$-system $\mathcal{S}$ is uniquely determined by the structure
$\bigcap_{Y\in\mathcal{S}}Y$; for $\mathcal{S}$ is the collection of
all copies $X$ of $C$ in $\Age(\mathrm{R})$ for which
$\bigcap_{Y\in\mathcal{S}}Y$ is an induced substructure of $X$.

Enumerate the members of ${\mathrm{R}\choose \C}$ as $C_n$, $n\in\om$,
and enumerate the maximal $\delta$-systems of copies of $C$ as
$\mathcal{S}_i$, $i\in\om$.  Let $\ell\in\om$ be fixed.  If there are
$k+1$ members of $\mathcal{S}_0$ with $f$-colour $\ell$, then choose
$n(0,\ell)$ to be the least $n$ such that $C_n\in \mathcal{S}_0$ and
$f(C_n)=\ell$, and let $N_{0,\ell}=\{n(0,\ell)\}$.  Otherwise, let
$N_{0,\ell}=\emptyset$.

Suppose we have chosen $N_{i,\ell}$. 
If either (a)
there are $k+1$ members
of $\mathcal{S}_{i+1}$ with $f$-colour $\ell$
and
 there is an
$n\in N_{i,\ell}$ such that $C_n\in\mathcal{S}_{i+1}$ and $f(C_n)=\ell$, or
(b) there are at
most $k$ members of $\mathcal{S}_{i+1}$ with $f$-colour $\ell$, then let
$N_{i+1,\ell}=N_{i,\ell}$.  Otherwise, there are $k+1$ members of
$\mathcal{S}_{i+1}$ with $f$-colour $\ell$ and none of these members have
an index in $N_{i,\ell}$.  Then let $n(i+1,\ell)$ be the least $n$ such that
$C_n\in\mathcal{S}_{i+1}$ with $f(C_n)=\ell$, and set
$N_{i+1,\ell}=N_{i,\ell}\cup\{n(i+1,\ell)\}$.  

 At the end of the inductive
construction, let $N_\ell=\bigcup_{i\in\om}N_{i,\ell}$.  Note that this
construction ensures that for each $i\in\om$ for which
$\mathcal{S}_{i}$ has $k+1$ members with $f$-colour $\ell$, there is
exactly one $n\in N_l$ such that $C_n\in\mathcal{S}_i$.  Note further
that for $\ell\ne \ell'$, $N_\ell\cap N_{\ell'}=\emptyset$.

Let $N=\bigcup_{l\in \om} N_l$ and 
define  a $k$-delta colouring $g$ on ${\mathrm{R}\choose \C}$ as follows.
For each $n\in\om\setminus N$,
let $g(C_n)=f(C_n)$;
and for $l\in N$,
let $g(C_n)=\om+n$.
Then $g$ is $k$-delta, since
if $\mathcal{S}_{i}$ has $k+1$ members with $f$-colour $l$, then 
there is exactly one $n\in N_l$ such that $C_n\in\mathcal{S}_i$, and $g(C_n)$ is defined to be $\om+n$.
By the induction hypothesis, 
there is a copy  $\mathrm{R}'\in {\mathrm{R}\choose\mathrm{R}}$ 
such that $g$ is one-to-one on ${\mathrm{R}'\choose\C}$.
Then  $f$ is $2$-delta on  ${\mathrm{R}'\choose \C}$.
By the  induction hypothesis applied again,  there is another copy $\mathrm{R}^*$ in  ${\mathrm{R}'\choose\mathrm{R}}$  such that $f$ is one-to-one on ${\mathrm{R}^*\choose\C}$.
\end{proof}

\begin{prop}\label{prop:gluing}
Let $m\leq n\in \omega$ and $m+2(n-m)\geq 2$, and let $\mathrm{C}$ be
a finite induced substructure of $\mathbb{D}(n,m)$. Then 
$$
\mathbb{D}(n,m)\underset{2\text{-delta}}{\xrightarrow{\mathrm{rainbow}}}\big(\mathbb{D}(n,m)\big)^\mathrm{C}.
$$
\end{prop}

\begin{proof}
Let $\mathfrak{d}=m+2(n-m)$ and $C\subseteq \mathfrak{D_d}$ with
$\mathrm{C}=\mathbb{D}(n,m)\down C$.  Let $\mathscr{B}$ be the set of
triples $(B, X,Y)$ of subsets of $\mathfrak{D_d}$ for which $B=X\cup
Y$ and $X\not=Y$ and $X\stackrel{s}{\sim}C\stackrel{s}{\sim}Y$. To
easily identify the order of the copies we require further that the
longest sequence $x\not\in X\cap Y$ is an element of $X$.  Write
$(B,X,Y)\stackrel{s}{\sim}(B',X',Y')$ if $B\stackrel{s}{\sim}B'$,  and moreover write 
 $(B,X,Y)\equiv(B',X',Y')$ if $B\stackrel{s}{\sim}B'$ and if
$f[X]=X'$ and $f[Y]=Y'$ for the (unique) strong similarity $f$ of $B$ to
$B'$. Note that if $g$ is a strong similarity function of
$\mathfrak{D_d}$ onto a subset of $\mathfrak{D_d}$ and if
$(B,X,Y)\equiv(B',X',Y')$ then
$(B,X,Y)\equiv(g(B),g(X),g(Y))\equiv(g(B'),g(X'),g(Y'))$. Also, if
$(B,X,Y)\equiv(B,X',Y')$ then $X=X'$ and $Y=Y'$.

Enumerate the finitely many $\equiv$ equivalence classes as
$\mathcal{P}_i$ for $i\in r\in \omega$.  Let $\mathscr{B}_0$ denote
the set of all $B$ for which there exist $X$ and $Y$ such that
$(B,X,Y)\in\mathscr{B}$.  
Observe that for every $B\in \mathscr{B}_0$ and every $i\in r$, either
there is no pair $(X,Y)$ of copies of $C$ with $(B,X,Y)\in
\mathcal{P}_i$, or else there exists exactly one such pair.

Let $i\in r$ and $(B,X,Y)\in \mathcal{P}_i$ be given.  We claim that
there are at least three distinct $(B_j,X_j,Y_j)$, $j\le 2$, such
that:
\begin{enumerate}
\item  each $(B_j,X_j,Y_j)\equiv (B,X,Y)$,  and hence are in $\mathcal{P}_i$; 
\item $X_0,X_1,X_2$ are distinct and form a $\delta$-system; 
\item and $Y_0=Y_1=Y_2$.
\end{enumerate}

The three distinct sets will be obtained by stretching some appropriate
elements. First since $B$ is a subset of $\mathfrak{D_d}$, $B$ is also a
subset of $\mathfrak{T_d}$ and is strongly diagonal.  Let $x$ denote
the longest element in $B$ with $x\not\in X\cap Y$; without loss of
generality we may assume that $x\in X$.
Now for $z\in B$, define $\Lambda_x(z)=$
\[\begin{cases}
 z \text{ if }  |z|<|x|;\\
 (z(0),\dots,z(|x|-1),z(|x|),z(|x|),z(|x|+1),\dots,z(|z|-1)) \text{ if }  |z|> |x|. \\
\end{cases} \]
Now define $x^0 =x\concat{0}$, $x^1=x\concat{1}$, and $x^2=x$. For each $j\leq 2$,
let 
\[ \begin{cases}
B^j= \{\Lambda_x(z):z\in B\setminus \{x\}\}\cup\{x^j\} \\
X^j=\Lambda_x(X\setminus \{x\})\cup\{x^j\} \\
Y^j=\Lambda_x(Y). \\
\end{cases} \]
Then each $B^j$ is strongly diagonal, $B^j\stackrel{s}{\sim} B$, and
moreover, $(B^j,X^j,Y^j)\equiv (B,X,Y)$.  Let $B_j,X_j,Y_j$ denote
$\Delta_{\mathfrak{d}}(B^j), \Delta_{\mathfrak{d}}(X^j),\Delta_{\mathfrak{d}}(Y^j)$, respectively.
Since $\Delta_{\mathfrak{d}}$ is a strong diagonalization of
$\mathfrak{T_d}$, each $(B_j,X_j,Y_j)$, $j\in 2$ is in
$\mathcal{P}_i$.  The sets $X_0,X_1,X_2$ form a $\delta$-system of
copies of $C$.  It follows that for any strong similarity $f$, 
$f(\mathcal{P}_i)$ contains three different $\equiv$ related elements
with the same third entry whose three second entries form a
$\delta$-system of copies of $C$.

Let $\delta: \Sims_{\mathbb{D}(n,m)}(\mathrm{C})\to \omega$ be a
2-delta colouring.  Let $\lambda: \Sims_{\mathfrak{D_d}}(C)\to
\omega$ be given by $\lambda(C)=\delta(\mathrm{C})$ for
$\mathrm{C}=\mathbb{D}(n,m)\down C$. Then $\lambda$ is a 2-delta
colouring of $\Sims_{\mathfrak{D_d}}(C)$.  Associate with every $B\in
\mathscr{B}_0$ the $r$-tuple $\gamma(B)=(\sigma_i(B);i\in r)$, with
entries one of the elements in the set $\{=, \not=, \not\in\}$, so
that for $i\in r$:
\[   
\sigma_i(B)=  
\begin{cases}
 =    & \text{if $(B,X,Y)\in \mathcal{P}_i$ for  sets $X$ and $Y$ with $\lambda(X)=\lambda(Y)$  }, \\ 
\not=    & \text{if $(B,X,Y)\in \mathcal{P}_i$ for  sets $X$ and $Y$ with $\lambda(X)\not=\lambda(Y)$  }, \\ 
\not\in &\text{if there are no sets $X$ and $Y$ with $(B,X,Y)\in \mathcal{P}_i$}.
\end{cases}
\]
According to Theorem \ref{lem:beginbas} there exists a strong
similarity function $g$ of\/ $\mathfrak{D_d}$ to a tree $D\subseteq
\mathfrak{D_d}$ so that the colouring function $\gamma$ is constant on
every $\stackrel{s}{\sim}$-equivalence class of subsets in
$\mathscr{B}_0$.  Let $\mathscr{M}$ denote the set of all the copies
of $B$ in $D$.  

If there is an $i\in r$ and a $B\in \mathscr{M}$ for which
$\sigma_i(B)$ is equal to $=$, then there are sets $X$ and $Y$ with
$(B,X,Y)\in \mathcal{P}_i$ and with $\sigma_i(X)=\sigma_i(Y)$. By the
above the $\equiv$-equivalence class containing $(B,X,Y)$ contains (at least)
three elements $(B_j,X_j,Y_j)$, $j\le 2$, with the $X_j$, $j\le 2$,
distinct and moreover forming a $\delta$-system of copies of $C$. Then
$(g[B_j], g[X_j],g[Y_j])$, $j\le 2$, are three triples for which
$g[X_0], g[X_1], g[X_2]$ form a $\delta$-system of copies of $C$ with
$\lambda(g[X_0])=\lambda(g[X_1])=\lambda(g[X_2])$, contradicting that
$\lambda$ is 2-delta.

Thus $\sigma_i(B)$ is not equal to $=$ for every $B\in
\mathscr{M}$ and every $i\in r$.  Now suppose $X$ and $Y$
are copies of $C$ in $D$, the $g$-image of $\mathfrak{D_d}$.  Letting
$B=X\cup Y$, the triple $(B,X,Y)$ is in $\mathcal{P}_i$ for some $i\in
r$, and thus $\lambda(X) \neq \lambda(Y)$. 

This completes the proof. \end{proof}
 
We now observe that (delta) rainbow Ramsey is preserved among equimorphic structures.

\begin{lem}\label{lem:siblingsrain}  
Let $\mathrm{A}$ and $\mathrm{B}$ be two equimorphic structures. Then
\[
\mathrm{A}\underset{k\text{-delta}}{\xrightarrow{\mathrm{rainbow}}}\big(\mathrm{A}\big)^\mathrm{C}  \text{\, \, \,  if and only if\, \, \,  }  \mathrm{B}\underset{k\text{-delta}}{\xrightarrow{\mathrm{rainbow}}}\big(\mathrm{B}\big)^\mathrm{C}.  
\]
Similarly for the rainbow Ramsey property.
\end{lem}

\begin{proof}
By symmetry it suffices to assume that $\mathrm{A}$ is an
induced substructure of $\mathrm{B}$. 

Let $\mathrm{A}\underset{k\text{-delta}}{\xrightarrow{\mathrm{rainbow}}}\big(\mathrm{A}\big)^\mathrm{C}$
and $\gamma: \binom{\mathrm{B}}{\mathrm{C}}\to \omega$ be
$k$-delta. The restriction of $\gamma$ to
$\binom{\mathrm{A}}{\mathrm{C}}$ is $k$-delta and hence there a
copy $\mathrm{A}^\ast\in \binom{\mathrm{A}}{\mathrm{A}}$ so that
$\gamma$ is one to one on $\binom{\mathrm{A}^\ast}{\mathrm{C}}$. Let
$\mathrm{B}^\ast\in \binom{\mathrm{A}^\ast}{\mathrm{B}}$.  Then
$\gamma$ is one to one on $\binom{\mathrm{B}^\ast}{\mathrm{C}}$.
\end{proof}

We now come to the main result of this paper. 

\begin{thm}\label{thm:final}
The binary ultrahomogeneous structure $\mathbb{U}(n,m)$ is (delta) rainbow
Ramsey for all $m$ and $n$ in $\omega$ with $m\leq n$.
 
If $m\leq n\in \omega$ and $m+2(n-m)\geq 2$, then the tree structure
$\mathbb{T}(n,m)$ and its diagonal substructure $\mathbb{D}(n,m)$ of
$\mathbb{T}(n,m)$ are (delta) rainbow Ramsey as well.  
\end{thm} 

\begin{proof}
If $m\leq n\in \omega$ and $m+2(n-m)\geq 2$, the Theorem follows from
Proposition \ref{prop:gluing} in conjunction with
Lemma~\ref{lem:2bdd_enough} and then Theorem \ref{thm:sibUT} in
conjunction with Lemma~\ref{lem:siblingsrain}.
 
The special cases $\mathbb{U}(0,0)$ and $\mathbb{U}(1,1)$ follow from the
standard Ramsey theorem in a similar but much simpler way as in the
proof of Proposition \ref{prop:gluing}; see \cite{CM}.  
\end{proof}
 
By a standard compactness, the general (delta) rainbow Ramsey result implies a  finite version. 

\begin{lem}\label{lem:infinite->finite}
Let $\mathrm{R}$ be a countably infinite relational structure for
which the relation
$\mathrm{R}\underset{k\text{-delta}}{\xrightarrow{\mathrm{rainbow}}}\big(\mathrm{R}\big)^\mathrm{C}$
holds. Then for each $\B\in\Age(\mathrm{R})$ there is an
$\A\in\Age(\mathrm{R})$ such that the relation
$\mathrm{A}\underset{k\text{-delta}}{\xrightarrow{\mathrm{rainbow}}}\big(\mathrm{B}\big)^\mathrm{C}$
holds. \\
The corresponding result holds for rainbow Ramsey.
\end{lem}

\begin{proof}
Let $(r_i;i\in \omega)$ be an $\omega$-enumeration of $R$ and for
$n\in \omega$ let $\mathrm{R}_n$ be the induced substructure of
$\mathrm{R}$ on the set $\{r_i\mid i\in n\}$. Let $\mathrm{B}\in
\Age(\mathrm{R})$ and assume for a contradiction that for every $n\in
\omega$ there is a $k$-delta colouring $\gamma_n:
\binom{\mathrm{\mathrm{R}_n}}{\mathrm{C}}\to \omega$ for which the
restriction of $\gamma_n$ to $\binom{\mathrm{B}^\ast}{\mathrm{C}} $ is
not one-to-one for every $\mathrm{B}^\ast\in
\binom{\mathrm{R}_n}{\mathrm{B}}$.  We will construct a $k$-delta
colouring $\gamma: \binom{\mathrm{R}}{\mathrm{C}}\to \omega$ so that
for all $\mathrm{B}^\ast\in \binom{\mathrm{R}}{\mathrm{B}}$ the
colouring $\gamma$ is not injective on
$\binom{\mathrm{B}^\ast}{\mathrm{C}}$; contradicting
$\mathrm{R}\underset{k\text{-delta}}{\xrightarrow{\mathrm{rainbow}}}\big(\mathrm{R}\big)^\mathrm{C}$.

Let $\mathfrak{U}$ be an ultrafilter on $\omega$ which contains all
co-finite subsets of $\omega$. For $\mathrm{C}^\ast\in
\binom{\mathrm{R}}{\mathrm{C}}$ and $\mathrm{C}^\diamond \in
\binom{\mathrm{R}}{\mathrm{C}}$ let $\mathrm{C}^\ast$ and
$\mathrm{C}^\diamond$ be equivalent, $\mathrm{C}^\ast\sim
\mathrm{C}^\diamond$, if $\{n\in \omega\mid
\gamma_n(\mathrm{C}^\ast)=\gamma_n(\mathrm{C}^\diamond)\}\in
\mathfrak{U}$.  Let $\gamma: \binom{\mathrm{R}}{\mathrm{C}}\to \omega$
be a function which is constant on every $\sim$-equivalence class and
assigns different colours to elements in different $\sim$-equivalence
classes.

We claim that the function $\gamma$ is delta. For assume that
$\mathcal{S}$ is a $\delta$-system of structures containing $k$
structures in $\binom{\mathrm{R}}{\mathrm{C}}$ and that $\gamma$ is
constant on $\mathcal{S}$. This is not possible because there is then
an $n\in \omega$ so that $\gamma_n$ is constant on $\mathcal{S}$.

Finally we must show that there is no $\mathrm{B}^\ast\in
\binom{\mathrm{R}}{\mathrm{B}}$ for which $\gamma$ is injective on
$\binom{\mathrm{B}^\ast}{\mathrm{C}}$. To do so assume otherwise there is such a
$\mathrm{B}^\ast$; but because $\binom{\mathrm{B}^\ast}{\mathrm{C}}$ is
finite, there is an $n\in\omega$ for which $\gamma_n$ is injective on
$\mathrm{B}^\ast$, a contradiction.
\end{proof}

\begin{cor}\label{cor:infinite->finite}
If $\mathrm{R}$ is a countably infinite (delta) rainbow Ramsey relational
structure, then the class of finite structures isomorphic to an induced
substructure of $\mathrm{R}$ is also (delta) rainbow Ramsey.
\end{cor} 

\section{Conclusion}

In this paper we proved rainbow Ramsey results for simple binary
relational structures, leaving open a wide spectrum of relational
structures. Thus we can easily ask the following. 

\begin{questions}
\begin{enumerate}
\item
Which homogeneous relational structures have the rainbow Ramsey property?
\item
Are there some ages which have the rainbow Ramsey property, but the \Fra \ limit does not?
Or are the two equivalent?
\end{enumerate}
\end{questions}

In particular we do not know whether the homogeneous triangle-free
graph has the rainbow Ramsey property, and we can ask for which finite
triangle-free graph $\C$, does the countable triangle-free homogeneous
graph have the rainbow Ramsey property.

In \cite{CM}, Csima and Mileti proved that the reverse mathematical strength of the Rainbow Ramsey Theorem for pairs of natural numbers is strictly weaker over RCA$_0$ than Ramsey's Theorem.  
We may ask similar questions for the Rado graph and other homogeneous relational structures, replacing Ramsey's Theorem with the existence of  canonical partitions in the sense of Corollary \ref{cor:beginbas}.

\begin{questions}
\begin{enumerate}   
\item
For homogeneous binary simple structures, is the statement of Corollary \ref{cor:beginbas} stronger over RCA$_0$  than the statement of Theorem \ref{thm:final}? 
\item
In particular, what is the  reverse mathematical strength of the Rado graph being rainbow Ramsey for $k$-bounded colorings of edges; 
 how does its strength compare with those of Ramsey's Theorem and  the Rainbow Ramsey Theorem for $k$-bounded colorings of pairs of natural numbers?
\end{enumerate}
\end{questions}

\bibliographystyle{amsplain}
\bibliography{references}

\end{document}